\documentclass[12pt,a4paper]{amsart}
\usepackage{amsmath,amsfonts,amssymb,amsthm,amscd}
\usepackage[english]{babel}
\usepackage{dsfont}

\input{comment.sty}
\includecomment{MM}
\includecomment{CL}

\renewcommand{\a}{\alpha}
\renewcommand{\b}{\beta}
\newcommand{\g}{\gamma}
\newcommand{\G}{\Gamma}
\renewcommand{\d}{\delta}
\newcommand{\D}{\Delta}

\renewcommand{\l}{\lambda}
\renewcommand{\L}{\Lambda}
\newcommand{\m}{\mu}

\renewcommand{\t}{\tau}

\newcommand{\e}{\varepsilon}

\newcommand{\E}{{\mathbb E}}

\renewcommand{\P}{\mathbb P}

\newcommand{\Z}{{\mathbb Z}}
\newcommand{\F}{{\mathbb F}}

\newtheorem{theorem}{Theorem}[section]
\newtheorem{proposition}[theorem]{Proposition}
\newtheorem{lemma}[theorem]{Lemma}
\newtheorem{corollary}[theorem]{Corollary}

\newtheorem{fact}[theorem]{Fact}

\theoremstyle{definition}

\theoremstyle{remark}
\newtheorem{remark}[theorem]{Remark}

\begin{document}

\title[Speed on solvable groups]
      {About the speed of random walks on solvable groups}
      \author{J\'{e}r\'{e}mie Brieussel}
      \date{12th May 2015}
\maketitle

\begin{abstract}
We show that for each $\l \in [\frac{1}{2}, 1]$, there exists a solvable group and a finitely supported measure such that the associated random walk has upper speed exponent $\l$.
\end{abstract}

\section{Introduction}

The speed of a random walk measures the average distance between the starting point and the position of the particle. When the step distribution has finite support, or more generally finite first moment, the speed reflects some geometric features of the space. For instance, the speed of a random walk on a group is related to the return probability (see Kesten \cite{K} and for recent developpments Gournay \cite{G}, Saloff-Coste-Zheng \cite{SCZ}), to the entropy of the random walk (Kaimanovich-Vershik \cite{KV}) and to the Hilbert compression (Naor-Peres \cite{NP}). In fact, linearity of the speed on a group is equivalent to non-triviality of the Poisson boundary (Kaimanovich-Vershik \cite{KV}). This naturally raises the question to understand what are the possible other, i.e. sublinear, behaviors of speed. 

The central limit theorem essentially asserts that the speed of symmetric random walks on abelian groups is comparable, up to multiplicative constants, to $\sqrt{n}$. This behavior, called diffusive, is obtained for a large class of groups, including nilpotent groups (Hebisch-Saloff-Coste\cite{HSC}), the lamplighter group on cyclic base with finite lamps, solvable Baumslag-Solitar groups and many polycyclic groups (Revelle \cite{R}, Thompson \cite{T}). Wreath products, also called lamplighters, and iterated wreath products of abelian groups provide a countable infinity of other behaviors (Erschler \cite{E1} \cite{E}). Precisely, on $\G_1=\Z$ and $\G_{i+1}=\Z \wr \G_i$, the speed of some random walk is comparable with $n^{1-\frac{1}{2^i}}$, and if $\Z$ is replaced by $\Z^2$ in the inductive definition, the speed is comparable with $n$ over iterated logarithms. Erschler has also showed that when the support of the step distribution generates an infinite group, the speed is at least diffusive (Lee-Peres \cite{LP}). Moreover, the speed can be arbitrarily close to $n$ (Erschler \cite{E3}), it can oscillate between two different powers of $n$ (Erschler \cite{E4}) and the amplitude of oscillations can be maximal between diffusive and linear (Brieussel \cite{B}). 

A very large class of speed behaviors has been obtained by Amir-Virag \cite{AV}. Namely for any function between $n^{\frac{3}{4}}$ and $n^\l$ for $\l<1$ (and sufficiently regular in a mild sense), there exists a group with a finitely supported measure for which the speed of the random walk is comparable with this function. Amir and Virag also conjecture the existence of speed functions comparable to $n^\l$ for any $\l$ between $\frac{1}{2}$ and $\frac{3}{4}$. The main result of this paper is in the direction of this conjecture, giving exemples of speed functions between $\sqrt{n}$ and $n^{\frac{3}{4}}$. Recall that the upper (resp. lower) speed exponent of a random walk is the quantity $\limsup \frac{\log \E|Z_n|}{\log n}$ (resp. $\liminf \frac{\log \E|Z_n|}{\log n}$).

\begin{theorem}\label{speed}
For any $\l \in [\frac{1}{2},1]$, there exists a finitely generated solvable group with a finitely supported symmetric measure such that the upper speed exponent of the associated random walk is $\l$.

More precisely, let $\l$ belong to the interval $[1-\frac{1}{2^i},1-\frac{1}{2^{i+1}}]$ for $i \geq 1$ integer, and let $\e(n)$ be a sequence tending to infinity. Then there exists a finitely generated $i+2$-solvable group with a finitely supported symmetric measure such that the speed $\E(Z_n)$ of the associated random walk oscillates between $n^{1-\frac{1}{2^i}}$ and $n^\l$ up to multiplicative factor $\e(n)$ in the sense that $n^{1-\frac{1}{2^i}} \leq \E(Z_n) \leq n^\l$ for all $n$ large enough, $\E(Z_n) \geq \frac{n^\l}{\e(n)}$ for infinitely many $n$ and $\E(Z_n) \leq \e(n) n^{1-\frac{1}{2^i}}$ for infinitely many~$n$.
\end{theorem}

 Combined with \cite {LP}, we deduce the

\begin{corollary}
A real number $\l$ is the upper speed exponent of a random walk on a (solvable) group with finite support step distribution if and only if $\l \in \{0\} \cup[\frac{1}{2},1]$.
\end{corollary}

There is nothing specific with the function $n^\l$, and the theorem is relevant for any function neglectible against $1-\frac{1}{2^{i+1}}$, for instance $n^\l (\log n)^\m$ when $\e(n)=\log \log n$. On the other hand, it is not clear if the lower exponent for solvable groups can be prescribed to a value other than $1-\frac{1}{2^i}$ (see Remark \ref{low}).

Contrary to the groups in \cite{AV} \cite{B} \cite{E3} and \cite{E4}, the groups of Theorem \ref{speed} are solvable. This shows that there are more behaviors of speed among solvable groups than suggested by the previous examples \cite{E1} \cite{E} \cite{R} \cite{T}. In particular, the speed of random walks on solvable groups can oscillate, and there are uncountably many pairwise non-comparable behaviors. It would be interesting to know if solvability implies some specific constraints on the speed function. Recall that there are uncountably many pairwise non-quasi-isometric solvable groups by Cornulier-Tessera \cite{CT}, and that virtual solvability is not a quasi-isometry invariant by Erschler \cite{E+}

The heart of Theorem \ref{speed} is the case $i=1$ when $\l \in [\frac{1}{2},\frac{3}{4}]$. These groups are obtained by "interpolating" between $\Z \wr D_2$ and $\Z \wr D_\infty$, where $D_{l}$ is the dihedral group of size $2l$. With usual "switch-walk-switch" step distribution, the first one has speed $\sqrt{n}$ and the second $n^{\frac{3}{4}}$ by Erschler \cite{E}. The idea is that the switch generators $a$ and $b$ of order two of the dihedral lamp-group in $\Z \wr D_l$ should be placed at positions far apart. This guarantees that this group locally "looks like" $\Z \wr D_2$. It is also asymptotically equivalent to it, hence has diffusive speed when $l$ is finite. However, letting $l$ tend to infinity, the speed tends to $n^{\frac{3}{4}}$ pointwise. By an intermediate value argument, we can find some $l$ with a scale where the speed is roughly $n^\l$. The argument is repeated infinitely many times by means of a diagonal product.

This construction is reminiscent of the technics of perturbation of automata groups used by Erschler \cite{E3}, Kassabov-Pak \cite{KP} and Brieussel \cite{B}. Theorem \ref{speed} could be obtained in the setting of section 7 in \cite{B} for the directed infinite dihedral automata group. However, the presentation given here avoids any reference to this theory, prefering the group $\Z$ rather than the action of $D_\infty$ on a tree.

The notion of diagonal product and local properties of marked groups are presented in section \ref{s2}. Some properties of random walks on lamplighter groups $\Z \wr D_l$ are studied in section \ref{s3}. Theorem \ref{speed} is proved in section \ref{s4}, dealing first with the main case $\frac{1}{2} \leq \l \leq \frac{3}{4}$.

\section{Local properties of marked groups}\label{s2}

\subsection{Diagonal product of marked groups}

Let $\G=\langle \g_1,\dots,\g_k\rangle$ and $F=\langle f_1,\dots,f_k\rangle$ be two $k$-generated groups. We consider the groups together with their fixed and ordered generating set. (Equivalently, we regard them as two quotients of the same free group $\F_k$ of rank $k$). Their diagonal product is the subgroup $\D$ of $\G \times F$ generated by the diagonal generating set :
$$\D=\langle (\g_1,f_1),\dots,(\g_k,f_k)\rangle< \G \times F. $$
For each word in $\F_k$, we denote by $w^\G$ (respectively $w^F$, $w^{\D}$) its evaluation in $\G$ (respectively $F$, $\D$).

Recall that $\G$ is a marked quotient of $\G'$ if the application $\G' \rightarrow \G$ sending $\g'_i$ to $\g_i$ induces a surjective group homomorphism (equivalently, if $w^{\G'}=e$ implies $w^\G=e$ for any $w$ in $\F_k$), and that the marked balls of radius $R$ in $\G$ and $\G'$ are equivalent when $w^{\G'}=e$ is equivalent to $w^\G=e$ for any $w$ of length $\leq 2R$ in $\F_k$. 

\begin{lemma}\label{dgen}
With the above notations, assume $F$ is a finite group. Then there exists $C_1,C_2 >0$ such that for any word $w$ in $\F_k$
$$ |w^\G| \leq |w^\D| \leq C_1 |w^\G|+C_2, $$
where norms are with respect to the above generating sets.

Moreover, the constants $C_1,C_2$ depend only on the finite group $F$ and locally on $\G$, in the sense that there exists $R>0$  such that if $\G'=(\g'_1,\dots,\g'_k)$ is another marked $k$-generated group with $\G$ a marked quotient of $\G'$ and $B_{\G'}(R) \simeq B_\G(R)$, then for any word $w$ in $\F_k$, we also have
$$ |w^{\G'}| \leq |w^{\D'}| \leq C_1 |w^{\G'}|+C_2, $$
where $\D'$ is the diagonal product of $\G'$ and $F$.
\end{lemma}

\begin{proof} The first inequality is easy. For $w^\D \in \D$, there exists a shortest word $W$ in $\F_k$ such that $l(W)=|w^\D|$ and $W^\D=w^\D$. As the generating set is diagonal, we deduce $W^\G=w^\G$, so $|w^\G| \leq l(W) = |w^\D|$.

In order to get the second inequality, we introduce the generating set $S=\{(\g_1,e),\dots, (\g_k,e),(e,f_1),\dots,(e,f_k)\}$ of $\G \times F$. Obviously,  for any $(\g,f) \in \G \times F$, we have 
$$ |\g| \leq |(\g,f)|_S \leq |\g|+\textrm{diam}(F). $$
The projection of $\D$ on the first factor $\G$ is onto, so $\D$ is a finite index subgroup of $\G \times F$, and we can find a transversal of the quotient with representatives of the form $a_i=(e,f_i)$, that is $f_1,\dots f_r$ in $F$ such that $\D (e,f_1) \sqcup \dots \sqcup \D(e,f_r)=\G\times F$.

Now the set $T$ of products of the form $a_isa_j^{-1}$ with $s$ in $S$ which belong to $\D$ is a generating set of $\D$ and we have $|\d|_T \leq |\d|_S$ for any $\d$ in $\D$. (Indeed, for $\d=s_1\dots s_n$, let $a_{i_0}=e$ and $\D a_{i_j}$ be the representative of $\D a_{i_{j-1}}s_j$, then $\d=a_{i_0}s_1 a_{i_1}^{-1}\dots a_{i_{n-1}}s_n a_{i_n}^{-1} a_{i_n}$, with $a_{i_n}=e$ when $\d$ belongs to $\D$.)

So each word in this generating set $T$ can be written as a word in the diagonal generating set $D=((\g_1,f_1),\dots,(\g_k,f_k))$. Thus $|\d|_D \leq C_1 |\d|_T$ for any $\d$ in $\D$ where $C_1$ is the maximal length of the shortest words in $D$ representing the generators in $T$. We conclude that for any $\d=(\g,f)$ in $\D$, we have $|\d|_D \leq C_1(|\g| +\textrm{diam}F)$.

To see that the constants are local (among groups admitting $\G$ as quotient), first observe that in order to find the transversal, we only need to know which couples $(e,f)$ in $\G \times F$ belong to $\D$. If $(e,f)$ belongs to $\D$, then there exists $w_{(e,f)}$ in $\F_k$ with $w_{(e,f)}^\D=(e,f)$, that is $w_{(e,f)}^\G=e$ and $w_{(e,f)}^F=f$. If $(e,f)$ does not belong to $\D$, then whenever $\G'$ admits $\G$ as marked quotient, there is no word such that $w^{\G'\times F}=(e,f)$ (otherwise, we would also get $w^{\G\times F}=(e,f)$ by taking quotient).

Secondly, to describe the generating set $T$, we only need to know which products $t=a_isa_j^{-1}$ (which are now seen as words in $S$) belong to $\D$. If $t$ belongs to $\D$, then there is a word $w_t$ with $w_t^\D=t$. If $t$ does not belong to $\D$, then whenever $\G'$ admits $\G$ as a quotient, there is no word $w$ such that $w^{\G' \times F}=t$ (otherwise, we would get $w^{\G \times F}=t$ by taking quotient). 

Finally, the lemma holds when $R$ is greater than the length of the words $w_{(e,f)}^\D$ and $w_t^\D$. Indeed, these relations and the fact that $\G'$ admits $\G$ as quotient guarantee that the transversal and the generating set $T$ are not changed (as set of words in $S$) when we replace $\G$ by $\G'$. And any generator in $T$ is still expressed by the same word in $D$, preserving $C_1$.
\end{proof}

\subsection{Local coincidence of lamplighter groups} Let $l$ belong to $\Z_{\geq1}\cup\{\infty\}$, we denote $D_l=\langle a,b | a^2= b^2= (ab)^l=1\rangle$ the dihedral group of size $2l$. 

We denote $B \wr L=B \ltimes \sum_B L$ the lamplighter group (or wreath product) with base-group $B$ and lamp-group $L$, where $\sum_B L$ is the set of functions from $B$ to $L$ neutral almost everywhere, with the shift action of $B$ on it. Its elements are couples $(g,f)$ with $g$ in $B$ and $f: B \rightarrow L$ of finite support. The function identically neutral on $B$ is denoted $\mathds{1}$, and the function taking value $l$ at $g$ and neutral elsewhere is denoted $l \d_g$.


For integers $k \geq 0$, $l \geq 2$ and $m\geq 1$, denote $\G=\G(k,l,m)$ the group $\Z/m\Z \wr D_l$ with marking generating set $(+1,\mathds{1}),(0,a\d_0),(0,b\d_k)$. The value $\infty$ is allowed for $l$ and $m$. For instance, the group $\G(0,2,\infty)$ is the lamplighter group $\Z \wr D_2$ with usual generating set.

\begin{lemma}\label{k-1/2}
For $l \in \Z_{\geq 2} \cup \{\infty\}$ and $2k<m$, the connected component containing the identity of the subset $\{(n,f) ||n| <k/2\}$ of the groups $\G(k,l,m)$ and $\G(0,2,\infty)$. In particular, the balls of radius $\frac{k-1}{2}$ in these two marked groups coincide.
\end{lemma}

\begin{proof}
Any word in the generators can be written in the form
\begin{eqnarray*} w &=& (i_1,\mathds{1}) (0,a^{\e_1}\d_0)(0,b^{\eta_1}\d_k)\dots (i_s,\mathds{1}) (0,a^{\e_s}\d_0)(0,b^{\eta_s}\d_k) \\ &=& (i_1+\dots+i_s,a^{\e_1}\d_{i_1}b^{\eta_1}\d_{i_1+k}
\dots a^{\e_s}\d_{i_1+\dots+i_s}b^{\eta_s}\d_{i_1+\dots+i_s+k}),
\end{eqnarray*}
using the relation $(i,\mathds{1})(0,b\d_k)(-i,\mathds{1})=(0,b\d_{i+k})$. The word $w$ represents an element $(i,f)$ of the wreath product. We obtain a word describing $f(x)$ by keeping in the product above only the factors $a^{\e_j}$ and $b^{\eta_{j'}}$ where $i_1+\dots+i_j=x$ and $i_1+\dots+i_{j'}+k=x$. The word $w$ represents the identity if and only if $i=i_1+\dots+i_s=0$ and $f(x)=e$ for all $x$ in $\Z$.

As $D_2\simeq \langle a \rangle \times \langle b \rangle$ is a product of two groups of order $2$, the $a$-lamps and the $b$-lamps are independant in $\Z \wr D_2$. Moreover in $\Z/m\Z \wr D_l$, as long as $|i_1+\dots+i_j |<k/2$ for all $j$, then for every $x$ in $\Z$, the lamp $f(x)$ takes values in $\{e,a,b\}$. Thus, the lamp function of $w$ in $\Z \wr D_2$ can be recovered from that in $\Z/m\Z \wr D_l$  and vice versa, just by shifting the $b$-valued part of the function. 
\end{proof}

\begin{remark}\label{2k}
The above proof also shows that for each $k \geq 0$ there is an isomorphism of marked groups between $\G(k,2,\infty)$ and $\G(0,2,\infty)$.
\end{remark}

\subsection{An infinite diagonal product of lamplighter groups with dihedral lamps}\label{diaginf} Let $k=(k_s)$, $l=(l_s)$ and $m=(m_s)$ be three sequences of integers (the sequences $l$ and $m$ are allowed to take value $\infty$) such that $2k_s <m_s$ and denote $\D=\D(k,l,m)$ the infinite diagonal product of the groups $\G(k_s,l_s,m_s)$, that is the subgroup of $\prod_{s}\Z/m_s\Z \wr D_{l_s}$, generated by the three sequences $\t=(\t_s)$, $\a=(\a_s)$ and $\b=(\b_s)$ where $\t_s=(+1,\mathds{1})$, $\a_s=(0,a_s\d_0)$ and $\b_s=(0,b_s\d_{k_s})$.

\begin{lemma}\label{coincide}
Let $\D=\D(k,l,m)$ be as above, and let $R < \frac{\min k -1}{2}$, then the balls of radius $R$ in $\D$ and $\G(0,2,\infty)$ coincide. Moreover, if $k$ is unbounded, $\G(0,2,\infty)$ is a quotient of $\D$.
\end{lemma}

\begin{proof}
The first statement follows directly from Lemma \ref{k-1/2}. It also implies that if $w^{\G(0,2,\infty)} \neq e$, then for $k_s$ large enough, $w^{\G(k_s,l_s,m_s)}\neq e$, hence $w^\D \neq e$, proving the second statement.
\end{proof}

\section{Random walks on lamplighter groups}\label{s3}

To the generating set $\t=(+1,\mathds{1}),\a=(0,a\d_0),\b=(0,b\d_{k})$ of the wreath product $\G=\G(k,l,m)=\Z \wr D_l$ between the infinite cyclic group and a dihedral group, we associate the "switch-walk-switch" generating set consisting of all words $u_1\t^\eta u_2$ with $u_i$ in $\{\a^{\e_1}\b^{\e_2},\b^{\e_1}\a^{\e_2}\}$,  and we denote $\m$ the measure obtained on it by equidistribution and independence of $\e_j$ taking value $0$ or $1$ and $\eta$ taking value $-1$ or $1$. We denote $Z_n$, or $Z_n^l$ when we want to stress the order of the dihedral group, the random walk of law $\m$ at time $n$. 

The lamplighter interpretation is that at each step the lamplighter first switches the lamp at his position by $a^{\e_1}$ and the lamp at distance $k$ on the right by $b^{\e_2}$, secondly he moves by $\eta$ and thirdly switches the lamp at his new position by $a^{\e_1'}$ and the lamp at distance $k$ on the right by $b^{\e_2'}$. Note that the order of $\a$ and $\b$ in $u_i$ has no influence unless $k=0$ and $l>2$.

\subsection{Infinite dihedral lamp-group}

\begin{lemma}\label{Dinfty}
There exists positive constants $c_1,c_2>0$ depending only on $k$ such that the switch-walk-switch random walk on $\G(k,\infty,\infty)=\Z \wr D_\infty$ satisfies for all $n$:
$$c_1 n^\frac{3}{4} \leq \E|Z_n^\infty| \leq c_2 n^\frac{3}{4}. $$
\end{lemma}

This lemma generalizes the estimation of the speed of switch-walk-switch random walk on $\Z \wr \Z$ due to Erschler \cite{E}. Replacing the group $\Z$ by $D^\infty$ is harmless, but taking care of $k >0$ requires to slightly adapt her argument. The present lemma would follow directly from \cite{E} if we knew that speed were invariant under change of generating set (which is an open question).  

\begin{proof} First observe that the Cayley graphs of $\Z$ with respect to $\{-1,1\}$ and that of $D_\infty$ with respect to $\{a,b\}$ coincide, as well as the simple random walks on these graphs. We deduce that when $\e_i$ are independent equidistributed in $\{0,1\}$, we have $\E|a^{\e_1}b^{\e_2}\dots a^{\e_{n-1}} b^{\e_n}|\asymp \sqrt n$.

Now the lamp $f(x)$ of $Z_n$ at $x$ can be expressed as a product of $a^{\e_j}$ and $b^{\e_{j'}}$, at times $j$ when $i_1+\dots+i_j=x$ (the lighter at $0$ visits $x$) and $j'$ when $i_1+\dots+i_{j'}=x$ (the lighter at $k$ visits $x$). This word can be rewritten by simplifying each sub-$a$-word and each sub-$b$-word (for instance $a^1b^1b^1a^0a^0$ should be rewritten as $a^1b^0a^0$ but not as $a^1$). 

We denote $t_x^{(n)}$ the length of the word obtained, which we call the local time at $x$. A key observation is that the local times increase slowlier when $k$ is bigger. This local time depends only on the random walk induced on $\Z$, as $t_x^{(n)}$ counts the number of alternations between times when the lighter at $0$ visits $x$ and times when the lighter at $k$ visits $x$ (i.e. when the lighter at $0$ visits $x-k$).

With this definition of local times, the proof of \cite{E} applies to the present setting as soon as we check that there exists $c>0$ such that for any $\e>0$ we have $\P(\sum_{x \in \Z} t_x^{(n)} \geq cn) \geq 1-\e$, for $n$ large. Indeed, this permits to get Erschler's Proposition 1 when $\e$ is small enough to use Lemma 2. The present lemma is then nothing but Lemma 3 in \cite{E} with $A=D_\infty$. 

But when the trajectory induced on $\Z$ is a path of length $2k$ from $0$ to $-k$ and back to $0$, we have $t_0^{(n)} \geq 2$, so the required sum increases by $2$ with positive probability with every multiple of $2k$. This is sufficient.
\end{proof}

\begin{remark}\label{1/2}
Arguments of the same type show that the speed of the switch-walk-switch random walk in $\Z \wr D_l$ with finite $l$ is asymptotically equivalent to $\sqrt{n}$ up to multiplicative constants depending only on $l$ and $k$. 
\end{remark}

\subsection{Comparing finite dihedral lamp-groups}

\begin{lemma}\label{2l}
There exists a universal constant $C>0$ such that for any integers $k$, $l$, $n$, we have
$$\E|Z_n^l| \leq \E|Z_n^{2l}| \leq C \E|Z_n^l| $$
with respect to the switch-walk-switch generating set.
\end{lemma}

Replacing with other equivalent length functions just modifies the inequalities by multiplicative constants, for instance with respect to the usual generating sets $\t=(+1,\mathds{1}),\a=(0,a\d_0),\b=(0,b\d_{k})$ by a factor $\leq 5$. The first inequality is obvious since it holds pointwise.

In a dihedral group $D_{2l}$, let $Y_t=a^{\e_1}b^{\e_2}a^{\e_3}b^{\e_4}\dots a^{\e_t}$ with $\e_i$ independent  taking value $0$ or $1$ with probability $\frac{1}{2}$ (the last letter is $b$ when $t$ is even). Denote $p_t(x)=\P(d(Y_t,\{e,a\})=x)$ the distribution of the distance between $Y_t$ and the edge $\{e,a\}$ in the Cayley graph of $D_{2l}$.

\begin{fact}
For each $t \in \Z^+$, the distribution $(p_t(0),\dots,p_t(l-1))$ is non-increasing.
\end{fact}

\begin{proof}
This is obvious for $t=0$, and we have for $t$ even $p_{t+1}(2x)=p_{t+1}(2x+1)=\frac{1}{2}(p_t(2x)+p_t(2x+1))$, and for $t$ odd $p_{t+1}(2x+1)=p_{t+1}(2x+2)=\frac{1}{2}(p_t(2x+1)+p_t(2x+2))$ and $p_{t+1}(0)=p_t(0)$ and $p_{t+1}(l-1)=p_t(l-1)$, so the fact follows by induction.
\end{proof}

As $d(Y_t,\{a,e\}) \leq d(Y_t,e)=|Y_t| \leq d(Y_t,\{a,e\})+1$, we deduce that (for $l \geq 4$)
\begin{eqnarray}\label{3/2}
\P\left(|Y_t| \in \left[\frac{3l}{2}, 2l\right]\right) \leq \P\left(|Y_t| \in \left[\frac{l}{2},\frac{3l}{2}\right]\right)
\end{eqnarray}

In order to prove Lemma \ref{2l}, we condition our samples by the random walk obtained by projection on the cyclic base-group. This fixes the local time function $t^{(n)}:\Z \rightarrow \Z^+$ introduced in the previous proof and the final position $i_n$. Lemma \ref{2l} reduces to finding a universal constant $C$ such that for any $t^{(n)}$ and $i_n$ we have 
\begin{eqnarray}\label{2lcond}
\E(|Z^{2l}_n||t^{(n)},i_n) \leq C\E(|Z^{l}_n||t^{(n)},i_n)
\end{eqnarray}
 
Under this conditioning, $Z^{2l}_n=(i_n,f_n^{2l})$ is described by the random products $f_n^{2l}(x)=a^{\e_1(x)}b^{\e_2(x)}\dots a^{\e_{t^{(n)}(x)}(x)}$ for $x \in \Z$ of fixed length $t^{(n)}(x)$, starting either by $a$ or by $b$ (and the latter is determined by the projection on the cyclic group), where all the $\e_i(x)$ for $i\in \Z^+$ and $x \in \Z$ are independent.

Then the word length $|Z^{2l}_n|$ is obtained as the solution of the following travelling salesman problem. The traveller starts at $0$ and ends at $i_n$. Meanwhile, he must for each $x \in \Z$ visit the sites $x$ and $x-k$ respectively $A(x)$ and $B(x)$ times, in order that $A(x)$ and $B(x)$ are respectively the number of $a$'s and $b's$ in a word representing $f_n^{2l}(x)$ and these visits to $x$ and $x-k$ have to be made in the order of appearance of the letters $a$ and $b$ in this word. However there is no constraint in the order of visits corresponding to distinct integers $x$ and $x'$.

The word length $|Z^l_n|$ is obtained by the same travelling salesman problem, where we replace $f_n^{2l}(x)$ by its projection $f_n^l(x)$ onto $D_l$. Inequality (\ref{2lcond}) is only true in average because of the existence of words where $f_n^{2l}(x)=(ab)^l$ which is very long in $D_{2l}$, but trivial in $D_l$. Since there is no dependence on the length $n$ of the random product $Z^{2l}_n$, we drop the subscript to ease notations. Our first aim is to obtain a good approximation of the solution of this salesman problem.

For each, $Z=(i,f)\in \Z \wr D_l$, we call range of $Z$ both the set $\{x|f(x)/\langle b \rangle \neq e\}\cup\{x-k|f(x)/\langle a \rangle \neq e\}\cup\{0,i\}$ and the difference between its maximum and minimum (the context will make clear which one we refer to). It always satisfies $|Z| \geq \textrm{Range}(Z)$. In fact, the range of $Z$ is the set of integers visited by the projection on $\Z$ of a minimal  word representing $Z$. 

We denote $L=L_Z$ the function on $\Z$ taking non-negative value $L(x)$ such that $b^{\e_1(x)}(ab)^{L(x)}a^{\e_2(x)}$ with $\e_i\in \{0,1\}$ is a minimal word representing $f(x)$. The length of $Z$ is approximated in terms of the function $L$ and the range as follows.

\begin{lemma}\label{max}
For any $l,k \geq 0$ and $Z=(i,f)$, we have
$$\sum_{x \in \Z} 2 \max(L(y)\mathds{1}_{[y-k,y)}(x)) \leq |Z| \leq \sum_{x \in \Z} 2 \max(L(y)\mathds{1}_{[y-k,y)}(x))+5\textrm{Range}(Z), $$
with respect to the switch-walk-switch generating set.
\end{lemma}

\begin{proof}
For each $y$ in $\Z$, the solution of the travelling salesman problem has to visit $y$ and $y-k$ at least $L(y)$ times and alternately. If $x$ belongs to $[y-k,y)$ the salesman has to cross at least $2L(y)$ times the edge right of $x$. This gives the first inequality. (The choice of $\e_i$ guarantees that any other word representing $f(x)$ contains $(ab)^{L(x)}$ as subword.) 

To show the second inequality, it is sufficient to describe a path in $\Z$ that permits to represent $Z$. The first part of the path consists in moving from $0$ to the maximum of the range, crossing on the way all sites $x-k$ where $\e_1(x) \neq 0$ (and turning on the lamps there appropriately). This has length less than twice the range. The last part of the path consists in moving from the minimum of the range to $i$, crossing on the way all sites $x$ where $\e_2(x)\neq0$ (and turning on the lamps), which also has length less than twice the range.

The central part of the path consists in moving from the maximum to the minimum of the range, switching on the way each lamp at $x$ on $(ab)^{L(x)}$, this can be done in less than $\sum_{x \in \Z} 2 \max(L(y)\mathds{1}_{[y-k,y)}(x))+\textrm{Range}(Z)$ steps. Let us prove it by induction on $\sum_{x\in \Z} L(x)$. If the sum is zero, the path just goes from the maximum to the minimum. In general, the salesman should first travel from the maximum of the range to the maximum $x_1$ of the support of $L$, then to the maximum $x_1'$ of the points $x$ such that $L(y)=0$ for all $y \in [x, x+k)$, then back to $x_1$, switching on  the $a$-lamps between $x_1$ and $x_1'+k$ and the $b$-lamps between $x_1'$ and $x_1-k$ (the intervalle $[x_1'+k,x_1]$ is the rightmost $k$-connected component of the support of $L$). The rest of the path is described by induction applied between $x_1$ and the minimum of the range and to $L'$ satisfying $L'(x)=L(x)-1$ whenever $x$ is in the rightmost $k$-connected component of the support of $L$.
\end{proof}

\begin{proof}[Proof of Lemma \ref{2l}] It is sufficient to prove inequality (\ref{2lcond}). First observe that if $\forall x \in \Z$, $t(x) \leq \frac{l}{2}$, then $|Z^{2l}|=|Z^l|$. Indeed, this assumption implies that $|f^{2l}(x)|\leq \frac{l}{2}$, so its image in the quotient $D_{2l} \rightarrow D_l$ also has $|f^{l}(x)|\leq \frac{l}{2}$. If a path $\g$ in $\Z$ is the projection of a switch-walk-switch word representing $Z^l$, it permits to write $f^l(x)$ at each $x$. Thus the same path permits to write one of its two preimages $f^{2l}(x)$ or $f^{2l}(x)(ab)^l$. Now as the second preimage has length $\geq \frac{3l}{2}$ it contains as a subword any reduced word of length $\leq \frac{l}{2}$, in particular $f^{2l}(x)$. Thus keeping the same projection $\g$ on $\Z$ and only modifying the lamp-factors, we can find a word of length $|\g|$ representing $Z^{2l}$. 

For the general case, let $X$ denote the set of points $x$ where $t(x) \geq \frac{l}{2}$. To each sample we associate $A=\{x\in \Z | |f^{2l}(x)| \geq \frac{l}{2}\}$ and we condition by subsets $A$ of $X$ to get 
\begin{eqnarray*}
\E(|Z^{2l}||t,i) =\sum_{A \subset X} \E(|Z^{2l}||t,i,A)\P(\{x\in \Z | |f^{2l}(x)| \geq \frac{l}{2}\}=A)
\end{eqnarray*}
From now on we fix $t,i,A$ and the value of $f^{2l}$ outside $A$. We define $Z^{2l}_{\textrm{cut}}=(i,f^{2l}_{\textrm{cut}})$ by $f^{2l}_{\textrm{cut}}(x)=f^{2l}(x)$ for $x \notin A$, and $f^{2l}_{\textrm{cut}}=(ab)^{\frac{l}{4}}$ of length $\frac{l}{2}$ for $x \in A$. We have $|Z^{2l}| \leq 9 |Z^{2l}_{\textrm{cut}}|$. Indeed, if we have an appropriate path for $Z^{2l}_{\textrm{cut}}$, up to going through back and forth four times, we can write any representative word at $x$ instead of $(ab)^{\frac{l}{4}}$, and a ninth travel ends at $i$. Moreover, our first observation shows that $|Z^{2l}_{\textrm{cut}}|=|Z^{l}_{\textrm{cut}}|$.

Now consider the random partition of $A$ into $A^+ \cup A^-$, where $A^+$ is the set of points $x$ in $A$ where $|f^{2l}(x)|\in [\frac{l}{2},\frac{3l}{2}]$. By (\ref{3/2}), we have for each $x$ in $A$ that $\P(x \in A^-|t,i,A) \leq \frac{1}{2}$. We define the discretized element $Z^l_{\textrm{rand}}$ by $f^l_{\textrm{rand}}(x)=f^l(x)$ for $x \notin A$, $f^l_{\textrm{rand}}(x)=(ab)^{\frac{l}{4}}$ for $x \in A^+$, and $f^l_{\textrm{rand}}(x)=e$ for $x \in A^-$. We have $|Z^l_{\textrm{rand}}| \leq 3 |Z^l|$. Indeed, $f^l(x)$ is a subword of $f^l_{\textrm{rand}}(x)$ except when $f^l(x)=(ba)^\frac{l}{4}$, so switching the lamps of an appropriate path for $Z^l$ is also appropriate except when $f^l(x)=(ba)^\frac{l}{4}$, but this word can be written at $x$ on the way back. A third travel ends at $i$. 

To get the lemma, there remains to show that 
\begin{eqnarray}\label{fin}
|Z^l_{\textrm{cut}}| \leq 22 \E(|Z^l_{\textrm{rand}}||t,i,A,f_{|\Z\setminus A})
\end{eqnarray}
By Lemma \ref{max}, we have 
\begin{eqnarray*}
\E(|Z^l_{\textrm{rand}}||t,i,A) \geq \E\sum_{x \in \Z} 2 \max_{y \in (x,x+k]}(L_{\textrm{rand}}(y)) \geq \sum_{x \in \Z} 2 \max_{y \in (x,x+k]}\E L_{\textrm{rand}}(y).
\end{eqnarray*}
Now for $y\notin A$, $L_{\textrm{rand}}(y)=L(y)=L_{\textrm{cut}}(y)$ and for $y\in A$, $L_{\textrm{rand}}(y)=\frac{l}{4}$ with probability $\P(f^{2l}(y) \in [\frac{l}{2},\frac{3l}{2}]) \geq \frac{1}{2}$ and $L_{\textrm{rand}}(y)=0$ with probability $\P(f^{2l}(y) \in (\frac{3l}{2},2l]) \leq \frac{1}{2}$ and $L_{\textrm{cut}}(y)=\frac{l}{4}$, so $\E L_{\textrm{rand}}(y) \geq \frac{1}{2}L_{\textrm{cut}}(y)$ for any $y$ in $\Z$. We deduce $\E(|Z^l_{\textrm{rand}}||t,i,A) \geq \sum_{x \in \Z}  \max(L_{\textrm{cut}}(y)\mathds{1}_{[y-k,y)}(x))$.

On the other hand, the range of $Z^l_{\textrm{rand}}$ is equal to the range of $Z_{\textrm{cut}}^l$ with probability $\geq \frac{1}{4}$ (when both the minimal and maximal lamps are on), so $\E(|Z^l_{\textrm{rand}}||t,i,A) \geq \E(\textrm{Range}(Z^l_{\textrm{rand}})|t,i,A)\geq \frac{1}{4} \textrm{Range}(Z_{\textrm{cut}}^l)$. The two previous inequalities and Lemma \ref{max} imply (\ref{fin}).

Piling up the inequalities between $Z^{2l}, Z^{2l}_{\textrm{cut}}, Z^l_{\textrm{cut}}, Z^l_{\textrm{rand}}$ and $Z^l$, we obtain Lemma \ref{2l} with $C=594$.
\end{proof}

\section{Proof of Theorem \ref{speed}}\label{s4}

\subsection{A baby case} The following proposition is the core of Theorem \ref{speed}. It is not used directly in the proof, but serves as a guideline.

\begin{proposition}
Let $C$ be the universal constant of Lemma \ref{2l}, and let $\l \in (\frac{1}{2},\frac{3}{4})$. For any $n_0 \in \Z_{\geq 0}$, there exists $k,l$ integers such that the speed of the switch-walk-switch random walk on $\G(k,l,\infty)$ satisfies 
\begin{enumerate}
\item $\forall n \leq n_0, \E|Z_n^{\G(k,l,\infty)}|=\E|Z_n^{\G(0,2,\infty)}|$,
\item $\forall n \geq 0, \E|Z_n^{\G(k,l,\infty)}| \leq n^\l$,
\item $\exists N \geq n_0, \E|Z_n^{\G(k,l,\infty)}|\geq \frac{N^\l}{C}$,
\item $\exists C_1>0, \E|Z_n^{\G(k,l,\infty)}| \leq C_1 \sqrt{n}$ for $n$ large enough.
\end{enumerate}
\end{proposition}

\begin{proof}
By Lemma \ref{coincide}, choose $k$ large enough so that the balls of radius $n_0$ in $\G(k,l,\infty)$ and $\G(0,2,\infty)$ coincide. This gives (1). Remark \ref{1/2} ensures (4). Now let $l$ be minimal such that there exists $N$ with $\E|Z_N^{\Z\wr D_{2l}}| >N^\l$ (this exists by Lemma \ref{Dinfty}). Minimality gives (2) and we obtain (3) by Lemma \ref{2l}.
\end{proof}

\subsection{The case $\l \in [\frac{1}{2},\frac{3}{4}]$}\label{main}

Let $\D=\D(k,l,m)$ be the infinite diagonal product of the groups $\G(k_s,l_s,m_s)$ as defined in section \ref{diaginf}. This group is $3$-solvable. We assume that the sequence $k=(k_s)$ tends to infinity, and that all values of $k,l,m$ are finite.

For each $s\in \Z_{\geq 1}$, denote $\L_s$ the diagonal product from $k=1$ to $s-1$ of these groups, which is a finite group, and $M_s$ the diagonal product from $s$ to infinity. The group $\D$ is the diagonal product $\L_s \times M_s$.

By Lemma \ref{k-1/2}, the group $M_s$ admits $\G(0,2,\infty)$ as quotient, and if $\frac{k_s-1}{2} >R$, the balls of radius $R$ in $\L_s \times M_s$ and in $\D_s=\L_s \times \G(0,2,\infty)$ coincide.

Let $\l \in [\frac{1}{2},\frac{3}{4})$ et $\e(n) \rightarrow \infty$. To prove Theorem \ref{speed}, it is sufficient to construct by induction on $s$ integers $k_1<\dots<k_{s-1}$, $l_1,\dots,l_{s-1}$ and $m_1,\dots,m_{s-1}$  such that the switch-walk-switch random walk on $\D_s=\L_s \times \G(0,2,\infty)$ satisfies:
\begin{eqnarray}\label{rec}\begin{array}{l}
 \forall n \geq 1, n^{\frac{1}{2}} \leq \E|Z_n^{\D_s}| \leq n^\l \textrm{ and }\\
 \exists n_1 < N_1<\dots < n_{s-1}<N_{s-1},  
 \E|Z_{n_i}^{\D_s}| \leq n_i^{\frac{1}{2}}\e(n_i) \textrm{ and } \E|Z_{N_i}^{\D_s}| \geq \frac{N_i^\l}{\e(N_i)},
\end{array}\end{eqnarray}
and that $\frac{k_s-1}{2}>N_{s-1}$, so the balls of radius $N_{s-1}$ in $\G_s$ and in $\G$ coincide. (To get the Theorem with $\l=\frac{3}{4}$, the function $n^\l$ should be replaced by $\frac{n^{\frac{3}{4}}}{\log n}$, which does not alter the proof.)

There is nothing to prove for $s=0$. Assume this is true for $s$, and construct $k_s,l_s,m_s$. First observe that as $\L_s$ is a finite group, Lemma \ref{dgen} gives two constants $C_1,C_2$ such that $\E|Z_n^{\G(0,2,\infty)}| \leq \E|Z_n^{\D_s}| \leq C_1 \E|Z_n^{\G(0,2,\infty)}|+C_2$. By Remark \ref{1/2}, we can choose $n_s$ such that
\begin{eqnarray}\label{cond5}
\forall n \geq n_s, \E|Z_n^{\D_s}|\leq n^{\frac{1}{2}}\e(n) \textrm{ and } \frac{n^\l-C_2}{CC_1}\geq \frac{n^\l}{\e(n)},
\end{eqnarray}
where $C$ is the universal constant of Lemma \ref{2l}.

Now we choose $k_s$ such that $\frac{k_s-1}{2} > n_s$, which guarantees by Lemma \ref{coincide} that the balls of radius $n_s$ in $\G(0,2,\infty)$ and in $M_s$ coincide (and in particular also their speed functions for $n \leq n_s$), and such that $\frac{k_s-1}{2}$ is bigger than the radius $R$ of Lemma \ref{dgen} with $F=\L_s$ and $\G=\G(0,2,\infty)$ so that the constants $C_1,C_2$ are valid for $\L_s \times M_s$ instead of $\D_s$.

For $l$ even, we consider the family of auxiliary groups $G_l=\L_s \times \G(k_s,l,\infty)$. As $l$ is even, these groups all admit $\G(0,2,\infty)$ as quotient, so Lemma \ref{dgen} applies to both $G_l$ and $G_{2l}$. Moreover, Lemma \ref{2l} relates the speed in $\G(k_s,l,\infty)$ and $\G(k_s,2l,\infty)$. Combining these results, we get
\begin{eqnarray} \label{6} \frac{\E|Z_n^{G_l}|-C_2}{C_1} \leq \E|Z_n^{G_{2l}}| \leq C_1 C \E|Z_n^{G_l}|+C_2. \end{eqnarray}

Now let $l_s$ be the minimal $l$ such that there exists $N=N_s$ with $\E|Z_N^{G_{2l}}|>N^\l$. This $l_s$ exists by Lemma \ref{Dinfty} and because the balls of radius $l$ in $\G(k_s,l,\infty)$ and $\G(k_s,\infty,\infty)$ coincide, so the speed functions $\E|Z_n^{\G(k_s,l,\infty)}|$ converge pointwise with $l$ to $\E|Z_n^{\G(k_s,\infty,\infty)}|$. By (\ref{6}) and the second condition (\ref{cond5}), we have 
\begin{eqnarray}\label{7}
\forall n \geq 1, \E|Z_n^{\L_s \times \G(k_s,l_s,\infty)}| \leq n^\l \textrm{ and } \E|Z_{N_s}^{\L_s \times \G(k_s,l_s,\infty)}| \geq \frac{N_s^\l}{\e(N_s)}
\end{eqnarray}
Now we choose $m_s>2N_s$ in order that the balls of radius $N_s$ in $\L_s \times \G(k_s,l_s,\infty)$, in $\L_s \times \G(k_s,l_s,m_s) \times \G(0,2,\infty)$ and in $\L_s \times \G(k_s,l_s,m_s) \times M_{s+1}=\D$ all coincide as soon as $\frac{k_{s+1}-1}{2}>N_s$.

\begin{remark}
The existence of $l_s$ is due to an intermediate value argument. In this viewpoint, Lemma \ref{2l} is a statement of continuity in $l$ of the speed function of $\G(k,l,\infty)$, uniform in $k$.
\end{remark}

\begin{remark}\label{low}
By choosing $n_s$ even larger at step (\ref{cond5}) in the previous proof, we can get arbitrarily large intervals where $\E|Z_n^\D|$ is close to $\sqrt{n}$. In contrast, we have no control on the width of intervals where $\E|Z_n^\D|$ is close to $n^\l$. Note that if we try to reduce $n_s$, for instance to obtain a bigger lower exponent, we lose control over the upper exponent because the second condition (\ref{cond5}) and thus (\ref{7}) may fail.
\end{remark}

\subsection{The general case} To obtain the case $i=2$ (i.e. $\l \in [\frac{3}{4},\frac{7}{8}]$), the idea is essentially to replace $\D(k,l,m)$ by $\Z \wr \D(k,l,m)$ and use Erschler \cite{E} that $\E|Z_n^{\Z \wr \D}|$ is comparable with $\sqrt{n \E|Z_n^\D|}$. However, this is known only under the hypothesis that $\E|Z_n^\D|$ is concave, which is not clear in our case. In practice, we get Theorem \ref{speed} by running the machinery of the previous proof, "interpolating" between $\Z \wr (\Z\wr D_2)$ and $\Z \wr (\Z \wr D_\infty)$.

For $i \geq 1$, let $\G_{i+1}(k,l,m)=\Z/m\Z \wr \G_i(k,l,m)$, where $\G_1(k,l,m)=\G_1(k,l,m)$ with switch-walk-switch generating set. As marked generating set for $\G_{i+1}$, we use the switch-walk-switch generating set where switch generators are the marked generators of $\G_i$ at position $0$.

\begin{lemma}\label{last} With the above notations, for any integers $i \geq 1$, $l \geq 0$ and $k \geq 0$ 
\begin{enumerate}
\item if $2k<m$, the balls of radius $\frac{k-1}{2}$ in $\G_i(k,l,m)$ and $\G_i(0,2,\infty)$ coincide.
\item the speed in $\G_i(k, \infty, \infty)$ is comparable with $n^{1-\frac{1}{2^{i+1}}}$.
\item the groups $\G_i(k, 2, \infty)$ and $\G_i(0, 2, \infty)$ are isomorphic as marked groups, and their speed is comparable with $n^{1-\frac{1}{2^{i}}}$.
\item there is a universal constant $C_i$ such that 
$$ \frac{1}{C_i} \E|Z_n^{\G_i(k,l,\infty)}| \leq \E|Z_n^{\G_i(k,2l,\infty)}| \leq C_i \E|Z_n^{\G_i(k,l,\infty)}|. $$
\end{enumerate}
\end{lemma}

\begin{proof}
This follows from Lemma \ref{k-1/2} and \ref{Dinfty}, Remarks \ref{2k} and \ref{1/2}, Lemma \ref{2l} and Erschler \cite{E}.
\end{proof}

The group with speed oscillating between $n^{1-\frac{1}{2^{i}}}$ and $n^\l$ will be a diagonal product $\D_i=\D_i(k,l,m)$ of $\G_i(k_s,l_s,m_s)$ with appropriate sequences $k=(k_s)$ tending to $\infty$, $l=(l_s)$ and $m=(m_s)$ with $m_s>2k_s$. The proof of section \ref{main} applies, replacing the references to previous sections by Lemma \ref{last}.

\textsc{\newline J\'er\'emie Brieussel \newline Universit\'e de Montpellier \newline Place E. Bataillon cc 051 \newline 34095 Montpellier, France} \newline
\textit{E-mail address:} jeremie.brieussel@univ-montp2.fr

\end{document}